\documentclass[11pt,a4paper,subeqn,reqno]{article}

\usepackage{graphicx}
\usepackage{amssymb}
\usepackage{epstopdf}

\usepackage{latexsym}

\usepackage{amsmath,amsthm}

\usepackage[mathscr]{eucal}
\usepackage[all]{xy}
\usepackage{mathrsfs}
\usepackage{authblk}

\usepackage{amsfonts}
\usepackage{amsmath}
\usepackage[hyperpageref]{backref}
\usepackage[dvipsnames,usenames]{color}
\usepackage{color}
\usepackage{shuffle}
\usepackage{colortbl}
\usepackage{makeidx}
\usepackage[all]{xy}
\usepackage{shuffle}

\allowdisplaybreaks
\usepackage[text={152mm,225mm},centering]{geometry}

\usepackage{mathtools}
\usepackage{dsfont}
\usepackage{url}
\usepackage{enumerate}

\usepackage[hyperindex=true, colorlinks=false]{hyperref}

\newtheorem{theorem}{Theorem}

\newtheorem{lemma}{Lemma}[section]

\newtheorem{corollary}[lemma]{Corollary}

\newtheorem*{THM_Dyn}{Theorem~\ref{Dynthm2}'}
\newtheorem*{THM_Kim}{Theorem~\ref{gotofull}'}

\theoremstyle{definition}

\newtheorem{example}[lemma]{Example}
\newtheorem{definition}[lemma]{Definition}
\newtheorem{definition-lemma}[lemma]{Definition-Lemma}
\newtheorem{definition-theorem}[lemma]{Definition-Theorem}
\newtheorem{assumption}[lemma]{Assumption}
\newtheorem{remark}[lemma]{Remark}

\newcommand{\A}{{\mathscr A}}
\newcommand{\Q}{\mathbb{Q}}

\newcommand{\unA}{1_\A}
\newcommand{\uncA}{1_\cA}
\newcommand{\est}{{\varnothing}}

\newcommand{\om}{\omega}
\newcommand{\Om}{\Omega}
\newcommand{\ii}{^{-1}}
\newcommand{\ti}{\tilde}
\newcommand{\pa}{\partial}

\newcommand{\cA}{{\mathcal{A}}}
\newcommand{\fD}{\mathfrak{D}}
\newcommand{\fN}{\mathcal{N}}
\newcommand{\fM}{\mathcal{M}}
\newcommand{\N}{\mathbb{N}}
\newcommand{\uN}{{\underline{\fN}}}
\newcommand{\UN}{{\underline{\fN}}}
\newcommand{\UM}{{\underline{\fM}}}
\newcommand{\un}{{\underline{n}}}
\newcommand{\ua}{{\underline{a}}}
\newcommand{\ub}{{\underline{b}}}
\newcommand{\uc}{{\underline{c}}}

\newcommand{\uOm}{{\underline{\Om}}}
\newcommand{\uom}{{\underline{\om}}}

\newcommand{\Bun}{B_{[\,\un\,]}}

\newcommand{\sig}{\sigma}
\newcommand{\kk}{\mathbf{k}}
\newcommand{\kUN}{\kk\,\UN}

\newcommand{\begla}{\begin{equation}}
\newcommand{\beglab}[1]{\begin{equation}	\label{#1}}
\newcommand{\edla}{\end{equation}}

\newcommand{\defeq}{\coloneqq}
\newcommand{\col}{\colon\thinspace}               

\newcommand{\ee}{\mathrm{e}}
\newcommand{\ens}{\enspace}
\newcommand{\ie}{{\it i.e.}\ }
\newcommand{\eg}{{\it e.g.}\ }

\newcommand{\idm}{{\mathds 1}}

\newcommand{\na}{\nabla}

\newcommand{\demi}{\frac{1}{2}}
\newcommand{\tiers}{\frac{1}{3}}
\newcommand{\quart}{\frac{1}{4}}

\newcommand{\imp}{\ens \Rightarrow \ens}
\newcommand{\IMP}{\; \Rightarrow \;}

\newcommand{\lhs}{{left-hand side}}
\newcommand{\rhs}{{right-hand side}}

\newcommand{\ad}{\operatorname{ad}}
\newcommand{\Ad}{\operatorname{Ad}}
\newcommand{\ord}{\operatorname{ord}}
\newcommand{\End}{\operatorname{End}}
\DeclareMathOperator{\LIE}{Lie}

\DeclarePairedDelimiter\abs{\lvert}{\rvert}%

\newcommand{\shh}[3]{\operatorname{sh}\!\big( \begin{smallmatrix}#1,\,
#2\\#3\end{smallmatrix} \big)}
\newcommand{\shabn}{\shh{\ua}{\ub}{\un}}

\newcommand{\SH}[3]{\operatorname{sh}\!\Big( \begin{smallmatrix}#1,\,
#2\\[1.5ex]#3\end{smallmatrix} \Big)}

\newcommand{\SHH}[3]{\operatorname{sh}\!\Big( \begin{matrix}#1,\,
#2\\[-.5ex]#3\end{matrix} \Big)}
\newcommand{\SHabn}{\SHH{\ua}{\ub}{\un}}

\newcommand{\wt}{\widetilde}

\begin{document}
\title{\textbf{The Baker-Campbell-Hausdorff formula \\[1ex] via mould calculus}}

\author[1]{Yong Li\thanks{Partially supported by NSFC (No.11131004, 11271269, 11771303), Email: yongli.math@hotmail.com}}
\author[2]{David Sauzin\thanks{Email: David.Sauzin@obspm.fr}}
\author[1]{Shanzhong Sun\thanks{Partially supported by NSFC (No.11131004, 11271269, 11771303), Email: sunsz@cnu.edu.cn}}

\renewcommand\Affilfont{\small}

\affil[1]{Department of Mathematics, Capital Normal University, Beijing 100048 P. R. China}
\affil[2]{CNRS UMR 8028 IMCCE, 77 av. Denfert-Rochereau, 75014 Paris, France}

\date{}


\maketitle

\begin{abstract}
  The well-known Baker-Campbell-Hausdorff theorem in Lie theory says
  that the logarithm of a noncommutative product $\ee^X \ee^Y$ can be
  expressed in terms of iterated commutators of~$X$ and~$Y$.
%
%
  This paper provides a gentle introduction to \'Ecalle's mould
  calculus and shows how it allows for a short proof of the above
  result, together with the classical Dynkin explicit formula
  \cite{ED} for the logarithm, as well as another formula recently
  obtained by T.~Kimura \cite{K17} for the product of exponentials
  itself.
  We also analyse the relation between the two formulas and indicate
  their mould calculus generalization to a product of more exponentials.
\end{abstract}

\setcounter{tocdepth}{1}\tableofcontents

\newpage


  \section{Introduction}

Let~$\A$ be a noncommutative associative algebra with unit.
In the associative algebra $\A[[t]]$ of all power series in an
indeterminate~$t$ with coefficients in~$\A$, one can take the
exponential of any series without constant term in~$t$ 
and the logarithm of any series with constant term~$\unA$. 
%
%
In this context, the famous Baker-Campbell-Hausdorff theorem (BCH theorem, for short)
can be phrased as
\beglab{BCHthm}
\log( \ee^{tX} \ee^{tY} ) \in \LIE(X,Y)[[t]]
\ens \text{for any $X,Y\in\A$,}
\edla
where $\LIE(X,Y)$ is the Lie subalgebra of~$\A$ generated by~$X$
and~$Y$, \ie the smallest subspace which contains~$X$ and~$Y$ and is stable
under commutator (see \eg \cite{BF} and references therein).

In fact,
using the notation $[A,B]$ or $\ad_A B$ for a commutator $AB-BA$, one has
\[
\log(\ee^{tX}\ee^{tY}) =
t(X+Y)+\frac{t^2}{2}[X,Y]+\frac{t^3}{12}([X,[X,Y]]+[Y,[Y,X]])
-\frac{t^4}{24}[Y,[X,[X,Y]]]+\cdots,
\]
where the coefficient of each power of~$t$ can be written in terms of
nested commutators involving~$X$ and~$Y$ only,
%
%
and there is a remarkable explicit formula due to Dynkin \cite{ED}:
\beglab{eqDynkForm}
\log(\ee^{X}\ee^{Y}) = \sum \frac{(-1)^{k-1}}{k}
\frac{t^{\sig}}{\sig}
\frac{[X^{p_1}Y^{q_1}\cdots X^{p_k}Y^{q_k}]}{p_1!q_1!\cdots p_k!q_k!}
\edla
  with summation over all $k\in\N^*$ and
  $(p_1,q_1),\cdots,(p_k,q_k)\in\N\times\N\setminus\{(0,0)\}$,
where $\sig \defeq p_1+q_1+\cdots+p_k+q_k$ and
$[X^{p_1}Y^{q_1}\cdots X^{p_k}Y^{q_k}] \defeq
\ad_X^{p_1}\ad_Y^{q_1}\cdots \ad_X^{p_k}\ad_Y^{q_k-1}Y$ if $q_k\ge1$
and
$\ad_X^{p_1}\ad_Y^{q_1}\cdots \ad_X^{p_k-1}X$ if $q_k=0$ (in which
case $p_k\ge1$).
Of course, the contribution of the terms with $q_k\ge2$, or with $p_k\ge2$
and $q_k=0$, is zero.

Our aim is to revisit the BCH theorem and the Dynkin formula in the
light of \'Ecalle's so-called ``mould calculus''.
We will show how mould calculus allows one to
prove these results with little effort, as well as an interesting formula
which was recently obtained by T.~Kimura \cite{K17} in relation to the BCH theorem and the
Zassenhaus formula and reads
\begin{multline}   \label{fullexpansion}
   \ee^{tX} \ee^{tY} = \unA +
\sum_{r=1}^\infty
\, \sum_{n_1,\ldots,n_r=1}^\infty \,
\frac{1}{n_r(n_r+n_{r-1})\cdots(n_r+\cdots+n_1)}
D_{n_1}\cdots D_{n_r}\\
\text{with $D_n\defeq \frac{t^n}{(n-1)!}\ad_{X}^{n-1}(X+Y)$ for each $n\ge1$.}
\end{multline}

We will also show how formula~\eqref{fullexpansion} and a little knowledge
of mould calculus immediately imply the BCH theorem, and how the
results can be generalized to a product of more than two exponentials.
%
%
It seems hard to prove all these facts using the methods of \cite{K17},
which rely on a lot of explicit combinatorial computations,
whereas almost no computation is needed when using a tiny
part of mould machinery.
In a nutshell, the point is that the rational coefficients
in~\eqref{fullexpansion} make up a ``symmetral mould''---in fact, a
very classical one in mould calculus---and that Dynkin's
formula~\eqref{eqDynkForm} is in essence a typical ``Lie mould expansion''
involving an ``alternal'' mould; we will explain in due time what
``mould expansions'', ``symmetrality'' and ``alternality'' are and how
they relate to the Lie theory.
We will also define a new operation in mould calculus, which gives the
relation between 
the rational coefficients appearing in
formulas~\eqref{eqDynkForm} and~\eqref{fullexpansion}.


Mould calculus was set up by J.~\'Ecalle in the 1980s as part of his
resurgence theory (\cite{E8185}, \cite{E2}).
Originally, \'{E}calle developed resurgence theory as a tool to study
analytic classification problems within dynamical system theory, first
for one-dimensional holomorphic germs, and then for much larger
classes of discrete dynamical systems or vector fields,
allowing him to tackle the Dulac conjecture about the finiteness of
limit cycles of planar analytic vector fields.
It soon turned out that resurgence theory has its own merits
not only in mathematics but also in physics.
For example, quantum resurgence was developed by \'{E}calle himself
(\cite{E84}) and Voros (\cite{V83}) to study the spectrum of
Schr\"{o}dinger operators, and it was continued by Pham and his
collaborators (\eg \cite{DDP93}) as an essential aspect of exact WKB
analysis.
The mathematical side of resurgence theory has evolved steadily
(\cite{S14}). Recently, resurgence theory has been at the forefront in
such diverse topics in mathematical physics as BPS spectrum
(\cite{GMN13}), supersymmetric field theories (\cite{BDSSU} and
references therein), resurgence and quantization as Riemann-Hilbert
correspondence (\cite{Kon17}), topological strings and Gromov-Witten
theory (\cite{CMS}, \cite{SSV16}), to name a few.

Resurgence theory deals with analytic functions which enjoy a certain
property of analytic continuation (``endlessly continuable
functions''),
which form an algebra,
and which typically appear as Borel transforms of certain divergent series.
In his systematic study of the singularities of these functions, their
monodromies and Stokes data, \'{E}calle 
%
discovered an infinite family of derivations acting on them, which
generate a free Lie algebra.
Mould calculus first appeared as a convenient combinatorial tool to
manipulate these derivations.
Later on, \'Ecalle also used mould calculus to study formal classification
problems in dynamical system theory, without any relation to
resurgence theory.
Mould calculus has since been used in various branches of mathematics,
for example
in the theory of multiple zeta values (\cite{EcaMZV}, \cite{Schneps}, \cite{BE17}, \cite{BS}),
in conjugacy problems for formal or analytic differential equations
\cite{M09}, \cite{S09},
%
in combinatorial Hopf algebras related to symmetric functions
\cite{JT},
in conjugacy problems in Lie algebras motivated by classical and
quantum dynamics \cite{PS16},
%
in the study of Rayleigh-Schr\"odinger series \cite{LMP}.

In the present paper, we do not assume any familiarity with mould
calculus on the part of the reader, and we introduce the most basic
ideas about moulds.
The BCH formula can be seen as an application, and we hope that
readers can find other interesting applications in mathematics or physics.


\newpage

The paper is organized as follows.
%
%
\begin{enumerate}[--]
\item Section~\ref{secMldCalc} 
  is a gentle introduction to mould calculus, containing the basic
  definitions and properties that we will require in our
  applications.
\item Section~\ref{secBCHD} 
  gives short proofs of the BCH theorem (Theorem~\ref{Dynthm1}) and
  Dynkin's formula (Theorem~\ref{Dynthm2}) based on mould calculus.
\item Section~\ref{secKimFullExp} 
  gives a short proof of Kimura's formula (Theorem~\ref{gotofull}) via
  mould calculus,
  as well as another derivation of the BCH theorem
  (Corollary~\ref{corKimFullExp}).
\item Section~\ref{secGenArbFact} 
  indicates how to generalize the previous results to the case of a
  product of more factors $\ee^{tX_1}\cdots \ee^{t X_N}$, with
  arbitrary $N\ge2$
(Theorems~\ref{Dynthm2}' and~\ref{gotofull}').
\item Section~\ref{secSigCompos} 
defines a new operation in mould calculus, that we call
$\sig$-composition, which allows us to relate the mould used for
Dynkin's formula and the one used for Kimura's formula.
\end{enumerate}


\section{Mould calculus for pedestrians}
\label{secMldCalc}


Throughout the article we use the notation
\[
\N = \{0,1,2,\ldots\}, \quad
\N^* = \{1,2,3,\ldots\}.
\]
%
In this section, we denote by~$\kk$ a field of characteristic zero (it
will be~$\Q$ in our later applications)
and by~$\fN$ a nonempty set
(in our applications, it will be either a finite set or~$\N^*$).

\subsection{The mould algebra}

Viewing~$\fN$ as an alphabet (the elements of which we call
``letters''), we denote by~$\uN$ the corresponding set of ``words''
(or ``strings''):
\[
\uN\defeq \{\un=n_1\cdots n_r\mid r\in \mathbb{N},\
n_1,\ldots,n_r\in\fN\}.
\]
The concatenation law
$(a_1\cdots a_r, b_1\cdots b_s)\in\uN\times\uN \mapsto
a_1\cdots a_r \, b_1\cdots b_s \in\uN$
yields a monoid structure, with the empty word~$\est$ as unit.

\begin{definition}
A $\kk$-valued mould on~$\fN$ is a function $\uN\to\kk$. The set of
all moulds is denoted by~$\kk^\uN$.
\end{definition}

Given a mould~$M$, it is customary to denote by~$M^\un$ the value it
takes on a word~$\un$.
%
%
%
\emph{Mould multiplication} is defined by the formula
\beglab{eqdefmouldmultiplic}
(M\times N)^{\un}\defeq \sum\limits_{(\ua,\ub) \ \text{such that} \ \un=\ua\,\ub}M^{\ua}N^{\ub}
\quad \text{for $\un\in\uN$,}
\edla
for any two moulds $M,N\in\kk^\uN$.
For instance,
\[ (M\times N)^{n_1n_2}=M^\est
N^{n_1n_2}+M^{n_1}N^{n_2}+M^{n_1n_2}N^\est. \]
It is immediate to check that \emph{$\kk^\uN$ is an associative
$\kk$-algebra}, noncommutative if~$\fN$ has more than one element,
whose unit is the mould~$\idm$ defined by
$\idm^\est = 1$ and $\idm^\un = 0$ for $\un\neq\est$.

We say that a mould~$M$ has order $\ge p$ if $M^\un=0$ for each
word~$\un$ of length $<p$. 
Clearly, if $\ord M\ge p$ and $\ord N\ge q$, then $\ord(M\times N)\ge
p+q$. In particular, if $M^\est=0$, then
$\ord M^{\times k}\ge k$ for each $k\in\N^*$, hence the moulds
\beglab{eqdefmouldexplog}
\ee^M \defeq \sum_{k\in\N} \frac{1}{k!} M^{\times k}
\ens\text{and}\ens
\log(\idm+M) \defeq \sum_{k\in\N^*} \frac{(-1)^{k-1}}{k} M^{\times k}
\edla
are well-defined (because, for each $\un\in\uN$, only finitely many
terms contribute to $(\ee^M)^\un$ or $(\log(\idm+M))^\un$).
We thus get mutually inverse bijections 
\[
\{\, M\in\kk^\uN \mid M^\est=0 \,\}
\quad \overset{\exp}{\underset{\log}{\rightleftarrows}} \quad
\{\, M\in\kk^\uN \mid M^\est=1 \,\}.
\]


\subsection{Comoulds and mould expansions}

Moulds are meant to provide the coefficients of certain multi-indexed
expansions in an associative algebra~$\cA$.
To deal with infinite expansions, we require this~$\cA$ to
be a complete filtered associative algebra, \ie there is an order
function $\ord \col \cA \to \N\cup\{\infty\}$ compatible with sum and
product,\footnote{%
We assume $\ord(A+B) \ge \min\{\ord A,\ord B\}$ and
$\ord(AB)\ge\ord A + \ord B$ for any $A,B\in\cA$,
and $\ord A = \infty$ iff $A=0$.
}
such that every family $(A_i)_{i\in I}$ of~$\cA$ is formally
summable provided, for each $p\in\N$,
all the $A_i$'s have order $\ge p$ except finitely many of them.
See \cite{S09} or \cite{PS16} for the details.
For the present paper, the reader may think of
\[ \cA = \A[[t]] \]
with the order function relative to powers of~$t$,
where~$\A$ is an associative algebra as in the introduction.

\begin{assumption}   \label{assumpBn}
  We suppose that we are given a family $(B_n)_{n\in\fN}$ in~$\cA$
  such that all the $B_n$'s have order $\ge1$ and, for each $p\in\N$,
  only finitely many of them are not of order $\ge p$.
\end{assumption}

\begin{definition}
We call \emph{associative comould generated by $(B_n)_{n\in\fN}$} the
family $(B_\un)_{\un\in\uN}$ defined by
$B_\est \defeq \uncA$ and
\[
B_{n_1\cdots n_r} \defeq B_{n_1} \cdots B_{n_r}
\quad \text{for all $r\ge1$ and $n_1,\ldots, n_r \in \fN$.}
\]
\end{definition}

\begin{lemma}    \label{lemMorphsmAssAlg}
The formula
\beglab{eqdefmldexp}
M \in \kk^\uN \mapsto
M B \defeq \sum_{\un\in\uN} M^\un B_\un \in \cA
\edla
defines a morphism of associative algebras.
Moreover,
\beglab{eqCompatExpLog}
M^\est = 0 \imp (\ee^M)B = \ee^{M B},
\qquad
M^\est = 1 \imp (\log M)B = \log(M B).
\edla
\end{lemma}

\begin{proof}
Observe that the family $(M^\un B_\un)_{\un\in\uN}$ is formally
summable in~$\cA$ thanks to our assumption on the $B_n$'s.
The property $B_{\ua\,\ub}=B_\ua B_\ub$ for all $\ua,\ub\in\uN$
entails
\beglab{eqMorphism}
(M\times N) B = (M B) (N B),
\edla
whence $M^{\times k} B = (M B)^k$ for all $k\in\N$,
and~\eqref{eqCompatExpLog} follows.
\end{proof}

It is the \rhs\ in~\eqref{eqdefmldexp} that is called a \emph{mould
expansion}.


\begin{example}   \label{exaBCHOm}
Suppose we are given $X,Y \in \A$, an associative algebra.
Take $\kk = \Q$, $\fN = \Om \defeq \{x,y\}$, a two-letter alphabet,
and $\cA = \A[[t]]$.
We then consider the associative comould generated by
\beglab{eqdeffamB}
B_x \defeq t X, \qquad B_y \defeq t Y.
\edla
Trivially, $t X = I_x B$ and $t Y = I_y B$, where $I_x, I_y \in
\Q^\uOm$ are defined by
\[
I_x^{\uom}\defeq \left\{ \begin{split}
    1\ens &\text{if $\uom$ is the one-letter word $x$}\\
    0\ens &\text{else,}
  \end{split} \right.
\qquad
I_y^{\uom}\defeq \left\{ \begin{split}
    1\ens &\text{if $\uom$ is the one-letter word $y$}\\
    0\ens &\text{else.}
  \end{split} \right.
\]
We thus get $\ee^{t X} = \ee^{I_x} B$, $\ee^{t Y} = \ee^{I_y} B$, and
\beglab{eqMldExpSOm}
\ee^{t X} \ee^{t Y} = S_\Om B
\ens\text{with $S_\Om \defeq \ee^{I_x} \times \ee^{I_y}$,}
\quad
\log( \ee^{t X} \ee^{t Y} ) = T_\Om B
\ens\text{with $T_\Om \defeq \log S_\Om$.}
\edla
By~\eqref{eqdefmouldmultiplic} and~\eqref{eqdefmouldexplog}, we get
\beglab{DynS}
S_\Om^\uom = \left\{ \begin{split}
  \frac{1}{p!q!} \quad &\text{if $\uom$ is of the form $x^py^q$ with $p,q\in\N$}\\
  0 \quad &\text{else,}
\end{split} \right.
\edla
thus the first part of~\eqref{eqMldExpSOm} is just another way of writing
$\ee^{t X} \ee^{t Y} = \sum \frac{t^{p+q}}{p!q!}X^p Y^q$.
\end{example}


In the general case, retaining from the associative algebra structure of~$\cA$ only the
underlying Lie algebra structure, \ie using only commutators (with the
notation $\ad_A B = [A,B]$), one can define another kind of mould expansion:

\begin{definition}
We call \emph{Lie comould generated by $(B_n)_{n\in\fN}$} the
family $(\Bun)_{\un\in\uN}$ of~$\cA$ defined by
$B_{[\est]} \defeq 0$ and
\[
B_{[n_1\cdots n_r]} \defeq
\ad_{B_{n_1}} \cdots \ad_{B_{n_{r-1}}} B_{n_r} =
[B_{n_1},[\cdots[B_{n_{r-1}},B_{n_r}]\cdots]].
\]
We define the \emph{Lie mould expansion} associated with a mould
$M \in \kk^\uN$ by the formula
\beglab{eqdefLiemldexp}
M [B] \defeq \sum_{\un\in\uN\setminus\{\est\}} \frac{1}{r(\un)} M^\un \Bun \in \cA,
\edla
where $r(\un)$ denotes the length of a word~$\un$.
\end{definition}



Division by~$r(\un)$ is just a normalization choice whose convenience will
appear in Section~\ref{secSymAlt}.
In Section~\ref{secBCHD}, we will prove the BCH theorem by showing how
to pass from the second part of~\eqref{eqMldExpSOm} to a Lie mould expansion.

\subsection{Symmetrality and alternality}    \label{secSymAlt}

One can get a morphism property for Lie mould expansions analogous
to~\eqref{eqMorphism} by imposing restrictions to the moulds
that we use: they must be ``alternal''. A tightly related
notion is that of ``symmetral'' mould.
The definition of both notions relies on word shuffling.

Recall that the shuffling of two words
$\ua = \om_1\cdots \om_\ell$ and
$\ub = \om_{\ell+1}\cdots \om_r$ is the set of
all the words~$\un$ which can be obtained by interdigitating the
letters of~$\ua$ and those of~$\ub$ while preserving their internal
order in~$\ua$ or~$\ub$, \ie the words which can be written
$\un = \om_{\tau(1)} \cdots \om_{\tau(r)}$
with a permutation~$\tau$ such that\footnote{%
Indeed, $\tau\ii(i)$ is the position in~$\un$ of~$\om_i$, the $i$-th letter of
$\ua\,\ub$.
}
$\tau\ii(1) < \cdots < \tau\ii(\ell)$
and $\tau\ii(\ell+1) < \cdots < \tau\ii(r)$.
We define the \emph{shuffling coefficient} $\shabn$ to be the number
of such permutations~$\tau$, and we set $\shabn\defeq0$ whenever~$\un$
does not belong to the shuffling of~$\ua$ and~$\ub$.
For instance, if $n,m,p,q$ are four distinct elements of~$\fN$,
\[
\SH{nmp}{mq}{nmqpm} = 0, \qquad
\SH{nmp}{mq}{mnqmp} = 1, \qquad
\SH{nmp}{mq}{nmmqp} = 2.
\]
We also define, for arbitrary words $\un$ and~$\ua$,
$\shh{\ua}{\est}{\un} = \shh{\est}{\ua}{\un} = 1$ if $\ua=\un$, $0$ else.

%

\begin{definition}
A mould $M\in\kk^\uN$ is said to be \emph{alternal} if $M^{\est}=0$ and
\beglab{eqdefalt}
\sum_{\un\in\uN} \SHabn M^{\un}=0
\quad \text{for any two nonempty words $\ua$, $\ub$.}
\edla
A mould $M\in\kk^\uN$ is said to be \emph{symmetral} if $M^{\est}=1$ and
\beglab{eqdefSymal}
\sum_{\un\in\uN} \SHabn M^{\un} = M^\ua M^\ub
\quad \text{for any two words $\ua$, $\ub$.}
\edla
\end{definition}

\begin{example}   \label{exaIxIyE}
It is obvious that any mould~$M$ whose support is contained in the set
of one-letter words (\ie $r(\un)\neq1 \Rightarrow M^\un=0$) is alternal.
For instance, the moulds~$I_x$ and~$I_y$ of Example~\ref{exaBCHOm} are alternal.
An elementary example of symmetral mould is~$E$ defined by
$E^\un\defeq \frac{1}{r(\un)!}$.
Indeed, since the total number of words obtained
by shuffling of any $\ua,\ub\in\UN$ (counted with multiplicity) is $\binom{r(\ua\,\ub)}{r(\ua)}$,
\[
\sum_{\un\in\uN} \SHabn E^{\un} =
\frac{r(\ua\,\ub)!}{r(\ua)!r(\ub)!}\cdot\frac{1}{r(\ua\,\ub)!}=E^{\ua}E^{\ub}.
\]
We shall see later that the moulds
$\ee^{I_x}$, $\ee^{I_y}$ and~$S_\Om$
involved in~\eqref{eqMldExpSOm} are symmetral,
and that $T_\Om$ is alternal.
\end{example}


In this paper,\footnote{%
In \'Ecalle's work, the initial motivation for the definition of alternality
and symmetrality is the situation when $\cA$ is an algebra of
operators (acting on an auxiliary algebra) and each~$B_n$ acts as a derivation:
in that case, the $\Bun$'s satisfy a modified Leibniz rule which
involves the shuffling coefficients, whence it follows that $MB$ is
itself a derivation if~$M$ is an alternal mould, and an algebra
automorphism if~$M$ is symmetral.
Here we do not assume anything of that kind on~$\cA$ and the $B_n$'s
but rather follow the spirit of ``Lie mould calculus'' as advocated
in \cite{PS16}.
}
we are interested in the shuffling coefficients because of the following
classical relation between the Lie comould and the associative comould:
\beglab{eqrelBunBn}
\Bun = \sum_{(\ua,\ub) \in \UN\times\UN} (-1)^{r(\ub)} r(\ua)
\, \SHabn \, B_{\,\wt\ub \, \ua}
\quad \text{for all $\un\in\uN$,}
\edla
where, for an arbitrary word $\ub = b_1\cdots b_s$, we denote
by~$\wt\ub$ the reversed word: $\wt\ub = b_s \cdots b_1$
(we omit the proof---see \cite{vonW}, \cite{CR}, \cite{PS16}).
An immediate and useful consequence is
\begin{lemma}\label{assliecomould}
%
%
If $M$ is an alternal mould, then $M[B] = MB$, \ie
%
\[
\sum_{\un\in\uN\setminus\{\est\}} \frac{1}{r(\un)} M^\un \Bun
= \sum_{\un\in\uN} M^\un B_\un.
\]
%
%
\end{lemma}
%

\begin{proof}
Putting together~\eqref{eqdefLiemldexp} and~\eqref{eqrelBunBn}, we get
$M[B] = \sum\limits_{\un\neq\est}\sum\limits_{\ua,\ub}
(-1)^{r(\ub)} \frac{r(\ua)}{r(\un)}
\, \shabn M^\un \, B_{\,\wt\ub \, \ua}$.
Now, $\shabn\neq0 \Rightarrow r(\un) = r(\ua) + r(\ub)$, hence
\[
M[B] = \sum_{r(\ua)+r(\ub)\ge1}
(-1)^{r(\ub)} \tfrac{r(\ua)}{r(\ua)+r(\ub)}
\bigg( \sum_{\un\in\UN} \, \SHabn M^\un \bigg) B_{\,\wt\ub \, \ua}
= \sum_{\ua\neq\est} M^\ua B_\ua = MB
\]
(the internal sum is~$M^\ua$ when $\ub=\est$
and it does not contribute when $\ua$ or $\ub\neq\est$ because of~\eqref{eqdefalt},
nor when $\ua=\est$ because of the factor $r(\ua)$).
\end{proof}


\emph{Any mould expansion associated with an alternal mould thus belongs to
the (closure of the) Lie subalgebra of~$\cA$ generated by the $B_n$'s},
since it can be rewritten as a Lie mould expansion, involving only
commutators of the $B_n$'s.

Lemma~\ref{assliecomould} is related to the classical
Dynkin-Specht-Wever projection lemma in the context of free Lie
algebras (see \eg \cite{CR}).
One should also mention that the concepts of symmetrality and
alternality are related to certain combinatorial Hopf algebras, as
emphasized by F.~Menous in his work on the renormalization theory in
perturbative quantum field theory---see \eg \cite{M09} and footnote~\ref{footHA}.
Hopf-algebraic aspects of mould calculus are also touched upon
in \cite{S09}, \cite{PS16} and \cite{LMP}.

For our applications, we require a last general result from mould
calculus (see \eg \cite{S09} for a proof):


\begin{lemma}\label{altsymmould}
$ $\\[-4.5ex]

\begin{itemize}
\item The product of two symmetral moulds is symmetral.
\item The logarithm of a symmetral mould is alternal.
\item The exponential of an alternal mould is symmetral.
\end{itemize}
\end{lemma}



\begin{example}
  The mould~$I$ defined by
\beglab{eqdefI}
I^\un = \left\{ \begin{split}
1 \quad &\text{if $r(\un)=1$}\\
  0 \quad &\text{else,}
\end{split} \right.
\edla
is alternal (being supported in one-letter words).
The symmetral mould~$E$ of Example~\ref{exaIxIyE} is $\ee^I$.
\end{example}

In fact, the set of all symmetral moulds is a group for mould multiplication,
the set of all alternal moulds is a Lie algebra for mould commutator,
and we get the analogue of~\eqref{eqMorphism} for Lie mould
expansions:
\[
\text{$M$, $N$ alternal} \imp
[M,N][B] = \big[ M[B], N[B] \big].
\]
Let us also mention a manifestation of the antipode of the Hopf algebra related to
moulds:\footnote{%
\label{footHA}
Denote by~$\kUN$ the linear span of the set of words, \ie the
$\kk$-vector space consisting of all formal sums $c = \sum c_\un \, \un$
with finitely many nonzero coefficients $c_\un\in\kk$.
Now, $\kUN$ is a Hopf algebra if we define multiplication by extending
$(\ua,\ub) \mapsto \ua\shuffle\ub \defeq \sum \shabn \un$ by bilinearity,
comultiplication by extending $\un \mapsto \sum\limits_{\un =
  \ua\,\ub} \ua\otimes\ub$ by linearity,
and antipode by extending $n_1\cdots n_r \mapsto (-1)^r n_r\cdots
n_1$ by linearity
(the unit is~$\est$ and the counit is $c\mapsto c_\est$).
The set of moulds can be identified with the set of linear forms on~$\kUN$,
if we identify $M\in\kk^\UN$ with $c \mapsto \sum M^\un c_\un$. 
The associative algebra structure~\eqref{eqdefmouldmultiplic}
of~$\kk^\UN$ is then dual to the coalgebra structure of~$\kUN$,
and alternal moulds appear as infinitesimal characters of~$\kUN$
(linear forms~$M$ such that $M(c\shuffle c')=M(c) c'_\est + c_\est M(c')$)
and symmetral moulds as characters
(linear forms~$M$ such that $M(\est)=1$ and $M(c\shuffle c')=M(c) M(c')$).
}
\[
\text{$M$ alternal} \IMP S(M) = - M, \quad
\text{$M$ symmetral} \IMP S(M) = \text{multiplicative inverse of~$M$,}
\]
where $S(M)^{n_1\cdots n_r} \defeq (-1)^r M^{n_r \cdots n_1}$.

All these facts are mentioned in \'Ecalle's works and can be proved by
Hopf-algebraic techniques or by direct computation.


  \section{The BCH Theorem and Dynkin's formula}   \label{secBCHD}

Let~$\A$ be an associative algebra. We now use mould calculus to prove


\begin{theorem}    \label{Dynthm1}
Suppose $X,Y \in \A$.
  Let $\Psi = \ee^{tX}\ee^{tY}\in\A[[t]]$.
Then
  \[ \log\Psi \in \LIE(X,Y)[[t]],\]
  where $\LIE(X,Y)$ is the Lie subalgebra of~$\A$ generated by~$X$
  and~$Y$.
\end{theorem}


\begin{theorem}[Dynkin, \cite{ED}]    \label{Dynthm2}
In the above situation,
  \begin{equation}\label{Dynpsi}
  \log\Psi = \sum \frac{(-1)^{k-1}}{k} \frac{t^{\sig}}{\sig}
\frac{[X^{p_1}Y^{q_1}\cdots X^{p_k}Y^{q_k}]}{p_1!q_1!\cdots p_k!q_k!}
  \end{equation}
  with summation over all $k\in\N^*$ and
  $(p_1,q_1),\cdots,(p_k,q_k)\in\N\times\N\setminus\{(0,0)\}$,
where $\sig \defeq p_1+q_1+\cdots+p_k+q_k$ and
$[X^{p_1}Y^{q_1}\cdots X^{p_k}Y^{q_k}] \defeq
\ad_X^{p_1}\ad_Y^{q_1}\cdots \ad_X^{p_k}\ad_Y^{q_k-1}Y$ if $q_k\ge1$
and
$\ad_X^{p_1}\ad_Y^{q_1}\cdots \ad_X^{p_k-1}X$ if $q_k=0$.
\end{theorem}


\begin{proof}[Proof of Theorem~\ref{Dynthm1}]
Half of the work has already been done in Example~\ref{exaBCHOm}!
With the two-letter alphabet $\Om=\{x,y\}$, $B_x = tX$ and $B_y = tY$,
we have $\log\Psi = T_\Om B$ with $T_\Om = \log S_\Om$,
$S_\Om = \ee^{I_x}\times\ee^{I_y}$.

The mould~$S_\Om$ is symmetral, because $I_x$ and~$I_y$ are alternal (they are supported in the set of
one-letter words) hence $\ee^{I_x}$ and~$\ee^{I_y}$ are
symmetral by Lemma~\ref{altsymmould}
and so is their product.
It follows, still by Lemma~\ref{altsymmould}, that~$T_\Om$ is
alternal.
Lemma~\ref{assliecomould} then shows that
\beglab{logS}
\log\Psi = T_\Om [B].
\edla
In particular, being expressed as a Lie mould expansion, $\log\Psi$
lies in $\LIE(X,Y)[[t]]$.
\end{proof}


\begin{proof}[Proof of theorem \ref{Dynthm2}]
With the same notation as previously, by definition,
%
\begin{gather*} 
T_\Om^{\uom} = \sum_{k\ge1} \tfrac{(-1)^{k-1}}{k}
\! \sum_{%
\substack{\uom^1\!, \ldots,\, \uom^k \in \uOm\setminus\{\est\} \\
\uom = \uom^1 \cdots \uom^k}}
S_\Om^{\uom^1} \cdots S_\Om^{\uom^k}
\quad \text{for each word~$\uom$,} \\[-2ex]
%
%
\intertext{hence \eqref{logS} yields}
\log\Psi = \sum_{k\ge1} \tfrac{(-1)^{k-1}}{k}
\sum_{\uom^1\!, \ldots,\, \uom^k \in \uOm\setminus\{\est\}}
\tfrac{1}{r(\uom^1) + \cdots + r(\uom^k)}
S_\Om^{\uom^1} \cdots S_\Om^{\uom^k}
B_{[\uom^1 \cdots \uom^k]}.
\end{gather*}
Inserting~\eqref{DynS}, we exactly get~\eqref{Dynpsi}.
\end{proof}


Mould calculus also allows us to express the inner derivation
associated with $\log\Psi$:

\begin{corollary}
The inner derivation of $\A[[t]]$ associated with $Z \defeq
\log(\ee^{tX} \ee^{tY})$ is
\beglab{eqinnerderBCH}
\ad_Z = \sum \frac{(-1)^{k-1} t^\sig}{k}
\frac{\ad_X^{p_1} \ad_Y^{q_1} \cdots \ad_X^{p_k} \ad_Y^{q_k}}{p_1!q_1!\cdots p_k!q_k!}
= \sum \frac{(-1)^{k-1}}{k} \frac{t^\sig}{\sig}
\frac{[ \ad_X^{p_1} \ad_Y^{q_1} \cdots \ad_X^{p_k} \ad_Y^{q_k} ]}{p_1!q_1!\cdots p_k!q_k!}
\edla
  with summation over all $k\in\N^*$ and
  $(p_1,q_1),\cdots,(p_k,q_k)\in\N\times\N\setminus\{(0,0)\}$,
where $\sig \defeq p_1+q_1+\cdots+p_k+q_k$ and
with the same bracket notation as in Theorem~\ref{Dynthm1}.
\end{corollary}

\begin{proof}
Working in the associative algebra $\End \A[[t]]$ with the comould and the Lie comould associated with
$A_x\defeq \ad_{tX}$ and $A_y\defeq\ad_{tY}$,
we get $\ad_Z = T_\Om [A]$ (\ie the second part
of~\eqref{eqinnerderBCH}) from~\eqref{Dynpsi} because
$A_{[\uom]} = \ad_{B_{[\uom]}}$.
Lemma~\ref{assliecomould} then entails $\ad_Z = T_\Om A$, \ie the
first part of~\eqref{eqinnerderBCH}
(which could have been obtained directly from
$\ad_Z = \log(\ee^{\ad_{tX}}\ee^{\ad_{tY}}$).
\end{proof}


  \section{Alternative formulas for $\ee^{tX} \ee^{tY}$ and its
    logarithm}
\label{secKimFullExp}


In this section, we take $\fN\defeq\N^*=\{1,2,3,\ldots\}$ as
our alphabet, and $\kk\defeq\Q$ as base field.
We now show how to find Kimura's formula~\eqref{fullexpansion} from
mould calculus.

\subsection{An alternative mould expansion for $\ee^{tX} \ee^{tY}$}
\label{secAltMldExpPsi}

\begin{theorem}[\cite{K17}]    \label{gotofull}
Let $X,Y\in\A$ as in Theorem~\ref{Dynthm1}.
Then $\Psi = \ee^{tX}\ee^{tY}$ can be written
\begin{align} \label{equfull}
\Psi &= \unA +
\sum_{r=1}^\infty
\, \sum_{n_1,\ldots,n_r=1}^\infty \,
\frac{1}{n_r(n_r+n_{r-1})\cdots(n_r+\cdots+n_1)}
D_{n_1}\cdots D_{n_r}\\
\label{eqdefDn}
&\qquad \text{with}\ens D_n\defeq \frac{t^n}{(n-1)!}\ad_{X}^{n-1}(X+Y) \quad\text{for each $n\ge1$.}
\end{align}
\end{theorem}

The rest of section~\ref{secAltMldExpPsi} is devoted to a new proof of this formula.



\begin{lemma}
$\Psi = \ee^{tX}\ee^{tY}$ is the unique element of $\A[[t]]$ such that
\beglab{equy}
\Psi_{\mid t=0} = \unA, \qquad t\pa_t \Psi = D \Psi,
\qquad\text{where}\ens
D \defeq t \, \ee^{tX} (X+Y)\, \ee^{-tX}.
\edla
\end{lemma}

\begin{proof}
The fact that~$\Psi$ satisfies~\eqref{equy} is straightforward.
On the other hand, if $\ti\Psi\in\A[[t]]$ is also solution
to~\eqref{equy}, then $\ord(\ti\Psi-\Psi)\ge1$ and it is easy to see
that in fact $\ord(\ti\Psi-\Psi)=\infty$ because 
$t\pa_t (\ti\Psi-\Psi) = D (\ti\Psi-\Psi)$ and
$\ord D\ge1$;
hence $\ti\Psi-\Psi = 0$.
\end{proof}


Let $\fN \defeq \N^*$ and consider the associative comould associated
with the family $(D_n)_{n\in\fN}$ defined by~\eqref{eqdefDn}. We have
\beglab{eqmouldexpID}
D = \sum_{n\in\fN} D_n = I D,
\edla
where~$D$ in the \lhs\ is the element of $\A[[t]]$ defined
in~\eqref{equy}, while the \rhs\ is the mould expansion associated
with the mould~$I$ defined by~\eqref{eqdefI}.
The proof of~\eqref{eqmouldexpID} is essentially the Hadamard lemma:
$\ad_X$ can be written $L_X-R_X$, where $L_X$ and~$R_X$ are the
operators of left-multiplication and right-multiplication by~$X$ and
\emph{they commute}, hence
$\ee^{t\ad_X} = \ee^{t(L_X-R_X)} = \ee^{tL_X} \, \ee^{-tR_X}$,
and $\ee^{tL_X}$ and $\ee^{-tR_X}$ are the operators of
left-multiplication and right-multiplication by $\ee^{tX}$ and $\ee^{-tX}$,
whence
\beglab{eqlemHadam}
\ee^{t\ad_X} A = \ee^{tX} A\, \ee^{-tX}
\quad\text{for any $A\in\A[[t]]$.}
\edla
In particular, $\ee^{tX} (X+Y) \ee^{-tX} = \sum_{n\in\fN} \frac{t^{n-1}}{(n-1)!}\ad_X^{n-1}(X+Y)$.


\begin{lemma}   \label{lemtpatna}
For any mould $S\in \Q^{\uN}$,
\[ t\pa_t (S D) = (\na S) D, \]
where $\na S$ is the mould defined by
\[
(\na S)^{n_1\cdots n_r} \defeq (n_1+\cdots+ n_r)S^{n_1\cdots n_r}
\quad\text{for each word $n_1\cdots n_r \in \UN$.}
\]
\end{lemma}

\begin{proof}
Obvious, since $D_n \in t^n \A$ for each $n\in\fN$.
\end{proof}


Lemma~\ref{lemMorphsmAssAlg}, formula~\eqref{eqmouldexpID} and
Lemma~\ref{lemtpatna} inspire us to look for a solution to~\eqref{equy}
in the form of a mould expansion:
\emph{$\Psi=SD$ will be solution to~\eqref{equy} if $S\in\Q^\UN$ is
  solution to the mould equation}
\beglab{eqmouldeq}
S^\est = 1, \qquad \na S = I \times S
\edla
(indeed: we have $(\na S)D = t\pa_t\Psi$ on the one hand, and
$(I \times S)D = (ID)(SD) = D \Psi$ on the other hand,
and $S^\est=1$ ensures $\ord(\Psi-\unA)\ge1$ because $\ord D_\un \ge1$
for all nonempty word~$\un$).
Now the second part of~\eqref{eqmouldeq} is equivalent to
\beglab{eqexplicitmldeqn}
(n_1+\cdots+ n_r) S^{n_1\cdots n_r}
= S^{n_2\cdots n_r}
\quad\text{for each nonempty word $n_1\cdots n_r \in \UN$},
\edla
thus the mould equation~\eqref{eqmouldeq} has a unique solution: the
mould $S_\fN \in \Q^\UN$ defined by
\beglab{symmeS}
S_\fN^{n_1\cdots n_r} \defeq \frac{1}{n_r(n_r+n_{r-1})\cdots(n_r+\cdots+n_1)}
\quad \text{for each $n_1\cdots n_r \in \UN$.}
\edla
In conclusion, $S_\fN$ is a solution to~\eqref{eqmouldeq},
thus $S_\fN D$ is a solution to~\eqref{equy},
thus 
\beglab{eqSfNDPsi}
S_\fN D = \Psi =\ee^{tX}\ee^{tY}
\edla
and formula~\eqref{equfull} is proved.



\begin{remark}
  For any alphabet~$\fN$ and base field~$\kk$, an arbitrary function
  $\phi \col \fN \to \kk$ gives rise to a linear operator
  $\na_\phi \col \kk^\uN \to \kk^\uN$ defined by the formula
\beglab{eqdefnaphi}
(\na_\phi\, M)^{n_1\cdots n_r} = \big(
\phi(n_1) + \cdots + \phi(n_r) \big) M^{n_1\cdots n_r}
\edla
(with the convention that an empty sum is~$0$).
The reader can check that \emph{$\na_\phi$ is a mould derivation}, \ie it
satifies the Leibniz rule
$\na_\phi(M\times N) = (\na_\phi\, M)\times N + M\times \na_\phi\, N$.
%
%
Here, we have used the mould derivation associated with the inclusion map
$\N^* \hookrightarrow \Q$.
\end{remark}

\subsection{An alternative Lie mould expansion for $\log(\ee^{tX}
  \ee^{tY})$}
\label{secAltLieMldExp}





The mould~$S_\fN$ that we have just constructed happens to be a very
common and useful object of mould calculus
(see \eg \cite{E8185} or \cite[\S13]{S09}).
It is well-known that it is symmetral; we give the proof for the sake
of completeness.

\begin{lemma}\label{exasym}
The mould~$S_\fN$ defined by the formula~\eqref{symmeS} is symmetral.
\end{lemma}

\begin{proof}
We prove the property~\eqref{eqdefSymal} for $M=S_\fN$ by induction on $r(\ua)+r(\ub)$.
The property holds when $\ua=\est$ or $\ub=\est$ because
$S_\fN^\est=1$. In particular it holds when $r(\ua)+r(\ub)=0$.

Suppose now that $\ua$ and~$\ub$ are arbitrary nonempty words.
Using the notation
\[
\abs\un \defeq n_1+\cdots+n_r, \quad
`\un \defeq n_2\cdots n_r
\quad \text{for any nonempty word $n_1\cdots n_r$,}
\]
we multiply the \rhs\ of~\eqref{eqdefSymal} by $\abs\ua + \abs\ub$:
we get
\beglab{eqabsuaubSuaSub}
(\abs\ua + \abs\ub) S_\fN^\ua S_\fN^\ub =
\abs\ua S_\fN^\ua S_\fN^\ub + \abs\ub S_\fN^\ua S_\fN^\ub =
S_\fN^{`\ua} S_\fN^\ub + S_\fN^\ua S_\fN^{`\ub} =
\sum_\uc\shh{`\ua}{\ub}{\uc} S_\fN^\uc + \sum_\uc\shh{\ua}{`\ub}{\uc} S_\fN^\uc,
\edla
where we have used~\eqref{eqexplicitmldeqn} and the induction hypothesis.
On the other hand, multiplying the \lhs\ of~\eqref{eqdefSymal} by $\abs\ua + \abs\ub$,
we get
\beglab{eqabsuaubsumun}
(\abs\ua + \abs\ub) \sum_\un \shabn S_\fN^\un =
\sum_\un \abs\un \shabn S_\fN^\un =
\sum_\un \shabn S_\fN^{`\un}
\edla
(using~\eqref{eqexplicitmldeqn} again).
The last sum can be split into two according to the first letter
of~$\un$, which must come either from the first letter of~$\ua$ or
from the first letter of~$\ub$ for $\shabn$ to be nonzero:
either $\un = a_1 \uc$ and $\shabn = \shh{`\ua}{\ub}{\uc}$,
or $\un = b_1 \uc$ and $\shabn = \shh{\ua}{`\ub}{\uc}$,
therefore~\eqref{eqabsuaubSuaSub} and~\eqref{eqabsuaubsumun} coincide,
which proves~\eqref{eqdefSymal} with $M=S_\fN$.
\end{proof}


We are now in a position to obtain a new formula for $\log\Psi$, on
which its Lie character is manifest---the new formula thus contains the BCH theorem:

\begin{corollary}    \label{corKimFullExp}
Let $T_\fN \defeq \log S_\fN \in \Q^\UN$. Then, with the notation of
Theorem~\ref{gotofull}, we have $\log\Psi = T_\fN [D]$, \ie
\[
\log(\ee^{tX} \ee^{tY}) = \sum_{r\ge1} \,\sum_{n_1,\ldots,n_r = 1}^\infty\,
\frac{1}{r} \, T_\fN^{n_1\cdots n_r} \,
[D_{n_1},[\cdots[D_{n_{r-1}},D_{n_r}]\cdots]]
\in \LIE(X,Y)[[t]].
\]
\end{corollary}

\begin{proof}
From Theorem~\ref{gotofull} and Lemma~\ref{lemMorphsmAssAlg} we deduce
\beglab{eqlogPsiTfND}
\log\Psi = \log(S_\fN D) = T_\fN D.
\edla
By Lemmas~\ref{altsymmould} and~\ref{exasym}, $T_\fN$ is alternal.
We conclude by Lemma~\ref{assliecomould}.
\end{proof}


From the definition
$T_\fN =
\sum\limits_{k=1}^\infty\frac{(-1)^{k-1}}{k}(S_\fN-\idm)^{\times k}$,
we can write down the coefficients for words of small length:
\[\begin{split}
T^{n_1}&=S^{n_1}=\frac{1}{n_1}\\
T^{n_1n_2}&=S^{n_1n_2}-\frac{1}{2}S^{n_1}S^{n_2}=\frac{n_1-n_2}{2n_1n_2(n_1+n_2)}\\
T^{n_1n_2n_3}&=S^{n_1n_2n_3}-\frac{1}{2}S^{n_1n_2}S^{n_3}-\frac{1}{2}S^{n_1}S^{n_2n_3}+\frac{1}{3}S^{n_1}S^{n_2}S^{n_3}\\
T^{n_1n_2n_3n_4}&=S^{n_1n_2n_3n_4}-\frac{1}{2}S^{n_1}S^{n_2n_3n_4}-\frac{1}{2}S^{n_1n_2}S^{n_3n_4}-\frac{1}{2}S^{n_1n_2n_3}S^{n_4}\\
&+\frac{1}{3}S^{n_1}S^{n_2}S^{n_3n_4}+\frac{1}{3}S^{n_1}S^{n_2 n_3}S^{n_4}+\frac{1}{3}S^{n_1 n_2}S^{ n_3}S^{n_4}-\frac{1}{4}S^{n_1}S^{n_2}S^{n_3}S^{n_4}
\end{split}
\]
\[\cdots\cdots\]
(omitting the subscript $\fN$ to lighten notation).
The low powers of~$t$ in $\log\Psi=T_\fN[D]$ can then be extracted from the Lie
mould expansion and we recover the classical BCH series:

\begin{align*}
\log\Psi &=\sum_{n_1=1}^\infty T^{n_1} D_{n_1}
+\sum_{n_1,n_2=1}^\infty \demi T^{n_1n_2} [D_{n_1},D_{n_2}]
+\sum_{n_1,n_2,n_3=1}^\infty \tiers T^{n_1n_2n_3} [D_{n_1},[D_{n_2},D_{n_3}]]\\
&+\sum_{n_1,n_2,n_3,n_4=1}^\infty \quart T^{n_1n_2n_3n_4} [D_{n_1},[D_{n_2},[D_{n_3},D_{n_4}]]]
+\cdots\\[1ex]
&=t(X+Y)+\frac{t^2}{2}[X,Y]+\frac{t^3}{3!}[X,[X,Y]]+\frac{t^4}{4!}[X,[X,[X,Y]]]+\frac{t^5}{5!}[X,[X,[X,[X,Y]]]]+\cdots\\
         &-\frac{t^3}{12}([(X+Y),[X,Y]])-\frac{t^4}{24}([(X+Y),[X,[X,Y]]])
           -\frac{t^5}{120}[[X,Y],[X,[X,Y]]] \\
         & \qquad \qquad \qquad \qquad \qquad \qquad \qquad \qquad \qquad \qquad \qquad
           -\frac{t^5}{80}[(X+Y),[X,[X,[X,Y]]]]+\cdots\\
&+\frac{t^5}{720}[(X+Y),[(X+Y),[X,[X,Y]]]]-\frac{t^5}{240}[[X,Y],[(X+Y),[X,Y]]]+\cdots\\
&+\frac{t^5}{720}[(X+Y),[(X+Y),[(X+Y),[X,Y]]]]+\cdots\\[1.5ex]
&=t(X+Y)+\frac{t^2}{2}[X,Y]+\frac{t^3}{12}([X,[X,Y]]+[Y,[Y,X]])-\frac{t^4}{24}[Y,[X,[X,Y]]]\\
         &-\frac{t^5}{720}[X,[X,[X,[X,Y]]]]
           -\frac{t^5}{720}[Y,[Y,[Y,[Y,X]]]]
           +\frac{t^5}{360}[X,[Y,[Y,[Y,X]]]] \\
         &+\frac{t^5}{360}[Y,[X,[X,[X,Y]]]]
           +\frac{t^5}{120}[Y,[X,[Y,[X,Y]]]]+\frac{t^5}{120}[X,[Y,[X,[Y,X]]]]
+ \cdots.
\end{align*}


  \section{Generalization to an arbitrary number of factors}
\label{secGenArbFact}

One of the merits of the mould calculus approach is that the formulas are easily
generalized to the case of
\[
\Psi = \ee^{tX_1} \cdots \ee^{tX_N} \in \A[[t]],
\]
where~$\A$ us an associative algebra and $X_1,\ldots,X_N\in\A$ for
some $N\ge2$.


\subsection{Mould expansion of the first kind}


\begin{THM_Dyn}
Let $\N^N_* \defeq \{\, p\in \N^N \mid p_1+\cdots+p_N \ge 1 \,\}$.
We have
\[
\log\Psi = \sum \frac{(-1)^{k-1}}{k} \frac{t^{\sig}}{\sig}
\frac{ \big[ X_1^{p^1_1} \cdots X_N^{p^1_N}
\cdots X_1^{p^k_1} \cdots X_N^{p^k_N} \big] }{p^1_1!\cdots p^1_N!
\cdots p^k_1!\cdots p^k_N!}
\]
  with summation over all $k\in\N^*$ and
  $p^1,\cdots,p^k\in\N^N_*$,
where $\sig \defeq \sum\limits_{i=1}^k \sum\limits_{j=1}^N p^i_j$ and
the bracket denote nested commutators as before.
\end{THM_Dyn}

\begin{proof}
Let $\Om\defeq \{x_1,\ldots,x_N\}$ be an $N$-element set. We consider
the associative comould generated by the family
\beglab{eqdefBxN}
B_{x_1} \defeq t X_1,\; \ldots, \; B_{x_N} \defeq t X_N \; \in \A[[t]].
\edla
We can write
$t X_1 = I_1 B,\; \ldots, \; t X_N = I_N B$, with moulds $I_1,\ldots,I_N \in \Q^\uOm$
defined by
\begin{gather}
\notag 
I_j^{\uom}\defeq \left\{ \begin{split}
    1\ens &\text{if $\uom$ is the one-letter word $x_j$}\\
    0\ens &\text{else}
\end{split} \right.
\\[-1ex]
\intertext{for $j=1,\ldots,N$.
Hence}
\label{eqPsiSOmBlogPsi}
\Psi = S_\Om B
\ens\text{with $S_\Om \defeq \ee^{I_1} \times \cdots \times \ee^{I_N}$,}
\quad \log\Psi = T_\Om B
\ens\text{with $S_\Om \defeq \log S_\Om$.}
\end{gather}
The moulds $I_1,\ldots,I_N$ are alternal (being supported in
one-letter words), hence Lemma~\ref{altsymmould} entails that their
exponentials are symmetral, and also~$S_\Om$, while~$T_\Om$ is
alternal.
We deduce that
\begin{gather*}
\log\Psi = T_\Om [B] =
\sum_{k\ge1} \tfrac{(-1)^{k-1}}{k}
\sum_{\uom^1\!, \ldots,\, \uom^k \in \uOm\setminus\{\est\}}
\tfrac{1}{r(\uom^1) + \cdots + r(\uom^k)}
S_\Om^{\uom^1} \cdots S_\Om^{\uom^k}
B_{[\uom^1 \cdots \uom^k]}.
\\[-1ex]
\intertext{The conclusion stems from the fact that}
S_\Om^\uom = \left\{ \begin{split}
  \frac{1}{p_1!\cdots p_N!} \quad &\text{if $\uom$ is of the form
    $x_1^{p_1}\cdots x_N^{p_N}$ with $(p_1,\ldots,p_N)\in\N^N$}\\
  0 \qquad\quad &\text{else.}
\end{split} \right.
\end{gather*}
\end{proof}


\subsection{Mould expansion of the second kind}


\begin{THM_Kim}
In the above situation,
$\Psi = \ee^{tX_1} \cdots \ee^{tX_N}$
can also be written
\begin{align} \label{equfullN}
\Psi &= \unA +
\sum_{r=1}^\infty
\, \sum_{n_1,\ldots,n_r=1}^\infty \,
\frac{1}{n_r(n_r+n_{r-1})\cdots(n_r+\cdots+n_1)}
\fD_{n_1}\cdots \fD_{n_r}\\[1ex]
\label{eqdefDnN}
&\qquad \text{with}\ens
\fD_n \defeq t^n \sum_{j=1}^N \,
\sum_{%
\substack{m_1, \ldots, m_{j-1} \in \N \\
m_1 + \cdots +m_{j-1} = n-1 }}
\frac{\ad_{X_1}^{m_1} \cdots \ad_{X_{j-1}}^{m_{j-1}}}{%
m_1! \cdots m_{j-1}!} X_j
\quad\text{for each $n\ge1$.}
\end{align}
\end{THM_Kim}


Note that formula~\eqref{equfullN} involves exactly the same rational
coefficients as in the case $N=2$.
The only difference in the formula is that the $D_n$'s
of~\eqref{eqdefDn} have been generalized to the $\fD_n$'s which are
defined in~\eqref{eqdefDnN} and read
\beglab{Nfullexpansionco}
\fD_n \defeq \left\{ \begin{aligned}
& t(X_1+\cdots+X_N) \; & \text{for}\ n=1,\\[1ex]
& t^n \frac{\ad^{n-1}_{X_1}}{(n-1)!}X_2 + \cdots +
t^n \sum\limits_{m_1+\cdots+m_{N-1}=n-1}
\frac{\ad^{m_1}_{X_1}\cdots \ad^{m_{N-1}}_{X_{N-1}}}{%
m_1!\cdots m_{N-1}!} X_N
\; & \text{for}\ n>1.
\end{aligned} \right.
\edla


\begin{proof}
We have $\Psi_{\mid t=0} = \unA$ and
\begin{align*}
t \pa_t \Psi &=
t X_1 \,\ee^{tX_1} \cdots \ee^{tX_N} +
t\,\ee^{t X_1} X_2 \,\ee^{t X_2}\cdots \ee^{t X_N} +\cdots
+ t\,\ee^{t X_1}\cdots \ee^{t X_{N-1}}X_N\,\ee^{t X_N} \\[1.5ex]
&= \fD \Psi,
\quad\text{where}\ens
\fD \defeq t \, \sum_{j=1}^N
\Ad_{\ee^{t X_1}} \cdots \Ad_{\ee^{t X_{j-1}}} X_j
\end{align*}
with the notation $\Ad_E A = E A E\ii$ for any $A \in \A[[t]]$
whenever~$E$ is an invertible element of $\A[[t]]$.
Moreover, we observe that there is no other solution in $\A[[t]]$ to
the system
\beglab{eqsystN}
\Psi_{\mid t=0} = \unA, \qquad
t\pa_t\Psi= \fD \Psi,
\edla
because $\ord\fD\ge1$.

Thanks to~\eqref{eqlemHadam}, we compute
$\fD = t \, \sum_{j=1}^N
\ee^{\ad_{t X_1}} \cdots \ee^{\ad_{t X_{j-1}}} X_j
= \sum_{n\ge1} \fD_n$.
Let us thus take $\fN=\N^*$ as alphabet and consider the associative
comould generated by $(\fD_n)_{n\in\fN}$, so that~$\fD$ can be
rewritten as the mould expansion $I\fD$, with the same mould as
in~\eqref{eqdefI}.

Lemmas~\ref{lemMorphsmAssAlg} and~\ref{lemtpatna} show that a mould expansion
$\Psi=SD$ is solution to~\eqref{eqsystN} if $S\in\Q^\UN$ is solution
to the mould equation~\eqref{eqmouldeq}
(indeed:
$(\na S)\fD = t\pa_t\Psi$ on the one hand, and
$(I \times S)\fD = (I\fD)(S\fD) = \fD \Psi$ on the other hand,
and $S^\est=1$ ensures $\ord(\Psi-\unA)\ge1$ because
$\ord \fD_\un \ge1$ for all nonempty word~$\un$).
But we already know that $S=S_\fN$ defined by~\eqref{symmeS} is the
unique solution to~\eqref{eqmouldeq},
hence 
\beglab{eqPsiSfNfD}
\Psi = S_\fN\fD,
\edla
which is equivalent to~\eqref{equfullN}.
\end{proof}


Notice that, in view of Section~\ref{secAltLieMldExp}, the
mould~$S_\fN$ is symmetral, the mould $T_\fN = \log S_\fN$ is
alternal, whence
\begin{gather}
\label{eqlogPsiTfNfD}
\log\Psi = T_\fN\fD = T_\fN[\fD],\\[-2ex]
\intertext{\ie}
\notag
\log(\ee^{tX_1} \cdots \ee^{tX_N}) = \sum_{r\ge1} \,\sum_{n_1,\ldots,n_r = 1}^\infty\,
\frac{1}{r} \, T_\fN^{n_1\cdots n_r} \,
[\fD_{n_1},[\cdots[\fD_{n_{r-1}},\fD_{n_r}]\cdots]]
\end{gather}
which thus belongs to $\LIE(X_1,\ldots,X_N)[[t]]$,
in accordance with the BCH theorem.


\section{Relation between the two kinds of mould expansion}
\label{secSigCompos}


In our application to products of two or more exponentials, we have
seen two different kinds of mould expansion.
The first kind involves an $N$-element alphabet $\Om\defeq
\{x_1,\ldots,x_N\}$ and the comould generated by the family $(B_\om)_{\om\in\Om}$ defined
by~\eqref{eqdefBxN}.
For the second one, the alphabet is $\fN\defeq \N^*$ and the comould is
generated by the family $(\fD_n)_{n\in\fN}$ which is defined
by~\eqref{Nfullexpansionco} and boils down to the $D_n$'s of~\eqref{eqdefDn}
when $N=2$.
A natural question is: What is the relation between both kinds of
mould expansion?
\ie can one pass from the representation of the
product~$\Psi$ as $S_\Om B$ in~\eqref{eqPsiSOmBlogPsi} to its
representation as $S_\fN\fD$ in~\eqref{eqPsiSfNfD}, or from
$\log\Psi = T_\Om B$ in~\eqref{eqPsiSOmBlogPsi}
to $\log\Psi = T_\fN\fD$ in~\eqref{eqlogPsiTfNfD}?

In this section, we will answer this question by defining a new
operation on moulds, which allows one to pass directly from~$S_\fN$
to~$S_\Om$, or from~$T_\fN$ to~$T_\Om$.
We take $N=2$ for simplicity but the generalization to arbitrary~$N$
is easy.

We start by giving a mould expansion of the first kind for the
$D_n$'s themselves.

\begin{lemma}
Let $\Om \defeq \{x,y\}$. The formula
\beglab{eqdefU}
\uom\in\uOm \mapsto
U^\uom \defeq \left\{\begin{split}
    1 \ens\; \quad & \text{if $\uom = x$} \\[1ex]
    \frac{(-1)^{q}}{p!q!} \quad & \text{if $\uom$ is of the form
      $x^pyx^q$ for some $p,q\in\mathbb{N}$}\\[1ex]
    0 \ens\; \quad & \text{else}
  \end{split}\right.
\edla
defines an alternal mould $U \in \Q^\uOm$ such that
\beglab{eqDnUnB}
D_n = U_n B \quad \text{for each $n\in\N^*$,}
\edla
where the \lhs\ is defined by~\eqref{eqdefDn} and the \rhs\ is the
mould expansion 
(for the comould generated by~\eqref{eqdeffamB})
associated with
\[ \text{$U_n \defeq $ restriction of~$U$ to the words of length~$n$.} \]
\end{lemma}

\begin{proof}
In view of~\eqref{eqMorphism}, we have
$\ad_{MB}(NB) = [MB,NB] = [M,N]B = (\ad_M N) B$
for any $M,N \in \Q^\uOm$, hence~\eqref{eqdefDn} can be rewritten as
$D_n = \frac{1}{(n-1)!} \ad_{I_xB}^{n-1}\big( (I_x+I_y)B \big) = U_n B$
with $U_n \defeq \frac{1}{(n-1)!} \ad_{I_x}^{n-1}(I_x+I_y)$.
Since $I_x$ and~$I_y$ are alternal and the set of all alternal moulds
is stable under mould commutator (as mentioned at the end of Section~\ref{secSymAlt}),
we see that this mould~$U_n$ is alternal.
Since the support of~$U_n$ is contained in the set of words of
length~$n$, the formula $U \defeq \sum_{n\ge1} U_n$ makes sense and
defines an alternal mould
(and $U_n$ now appears as the restriction of this~$U$ to the set of words of
length~$n$).
There only remains to check~\eqref{eqdefU}.

Now, $\ad_{I_x} = L-R$, where $L$ and~$R$ are the operators of
left-multiplication and right-multiplication by~$I_x$, which commute,
hence the binomial theorem yields
\[ 
U_n = \sum_{p+q=n-1} \tfrac{(-1)^{q}}{p!q!} L^p R^q (I_x+I_y)
= \sum_{p+q=n-1} \tfrac{(-1)^{q}}{p!q!} I_x^{\times p} \times (I_x+I_y)\times I_x^{\times q},
\]
\ie $U_n^\uom = 1$ if $\uom = x$ and $n=1$, 
$\frac{(-1)^{q}}{p!q!}$ if $\uom$ is of the form
$x^pyx^q$ for some $p,q\in\N$ such that $p+q=n-1$ (in which case $p$
and~$q$ are uniquely determined), and $0$ else.
Our~$U$ thus coincides with the mould defined by~\eqref{eqdefU}.
\end{proof}


In fact the proof just given shows that
\beglab{eqUAdeIx}
U = \ee^{\ad_{I_x}}(I_x+I_y) = 
\ee^{I_x} \times (I_x+I_y) \times \ee^{-I_x}.
\edla
This mould will allow us to relate $D$-mould expansions and
$B$-mould expansions:

\begin{theorem}   \label{thmodotcompo}
Let $\fN \defeq \N^*$. Define a linear map 
$M \in \Q^\UN \mapsto M\odot U \in \Q^\uOm$ 
by the formulas
\begin{align}  
\label{eqdefModotUest}
(M\odot U)^\est &\defeq M^{\est},\\[1ex]
\label{eqdefModotUuom}
(M\odot U)^\uom & \defeq \sum_{s \geq 1}\sum_{
\substack{
  \uom=\uom^1\cdots\,\uom^s\\ 
\uom^1,\ldots,\,\uom^s \in \uOm\setminus\{\est\}
}} 
M^{r(\uom^1)\cdots r(\uom^s)} U^{\uom^1}\cdots U^{\uom^s}
  \quad \text{for $\uom\in\uOm\setminus\{\est\}$.}\\[-2ex]
\intertext{Then}  \notag
& \qquad \quad \ens MD=(M\odot U)B
\quad\text{for any $M \in \Q^\UN$.}
\end{align}
\end{theorem}


Recall that $r\col \uOm \to \N^*=\fN$ is our notation for the length
function. In~\eqref{eqdefModotUuom},
$r(\uom^1)\cdots r(\uom^s)$ is to be understood as a word of length~$s$
of~$\UN$
(and the sum is finite because the words~$\uom^j$ are nonempty, hence
$s\le r(\uom)$).


\begin{proof}
By direct computation, using~\eqref{eqDnUnB} to express $D_{n_1\cdots
  n_s} = D_{n_1}\cdots D_{n_s}$,
  \begin{align*}
  MD &=\sum_{\un\in\UN} M^\un D_\un
  =M^\est \, \uncA + \sum_{s\geq 1}\,\sum_{n_1,\ldots,n_s\in\fN} M^{n_1\cdots n_s}
  %
  %
  \sum_{ \substack{ \uom^1 \in\uOm \\ r(\uom^1)=n_1 }}
  U^{\uom^1}B_{\uom^1}
 %
 %
 \cdots
 %
 %
 \sum_{ \substack{ \uom^s \in\uOm \\ r(\uom^s)=n_1 }}
  U^{\uom^s} B_{\uom^s}
 %
 %
\\[1ex]
  &= M^{\est}\,\uncA+\sum_{s\geq 1}\,\sum_{n_1,\ldots,n_s\in\mathcal{N}}M^{n_1\cdots n_s}
  \sum_{ \substack{ \uom^1,\ldots,\,\uom^s \in\uOm \\
  r(\uom^1)=n_1, \ldots, r(\uom^s)=n_s }}
  U^{\uom^1}\cdots U^{\uom^s} B_{\uom^1 \cdots \,\uom^s}  \\[1ex]
  &= M^{\est}\,\uncA+\sum_{s\geq 1}\,\sum_{\uom^1,\ldots,\,\uom^s\in\uOm\setminus\{\est\}}
   M^{r(\uom^1)\cdots r(\uom^s)}
  U^{\uom^1}\cdots U^{\uom^s} B_{\uom^1 \cdots \,\uom^s}   \\[1ex]
  &=M^{\est}\,\uncA+ 
\sum_{\uom \in \uOm\setminus\{\est\}}
\Bigg(
\sum_{ \substack{
s\ge1, \, \uom^1, \ldots\,\uom^s \in \uOm\setminus\{\est\} \\
\uom=\uom^1\cdots\,\uom^s }}
  M^{r(\uom^1)\cdots r(\uom^s)}
  U^{\uom^1}\cdots U^{\uom^s}
\Bigg)
  B_{\uom}
= (M\odot U)B.
  \end{align*}
\end{proof}


The relations $S_\fN D = S_\Om B$ (which coincides with~$\Psi$
according to~\eqref{eqMldExpSOm} and~\eqref{eqSfNDPsi})
and $T_\fN D = T_\Om B$ (which coincides with~$\log\Psi$
according to~\eqref{eqSfNDPsi} and~\eqref{eqlogPsiTfND})
now appear as a manifestation of Theorem~\ref{thmodotcompo} and the following


\begin{theorem}   \label{thmSNodotUSOmT}
\[ 
S_\fN\odot U = S_\Om, \qquad
T_\fN\odot U = T_\Om.
\]
\end{theorem}


The proof of Theorem~\ref{thmSNodotUSOmT} is given at the end of
this section.


Our definition \eqref{eqdefModotUest}--\eqref{eqdefModotUuom} of the mould operation~`$\odot$' is a variant of
\'{E}calle's mould composition~`$\circ$' which is defined for any
alphabet that is a commutative semigroup
(\cite{E84}, \cite{S09}, \cite{FFM}).
Here is a definition which encompasses both operations:

\begin{definition}
  Given two alphabets $\Om$ and~$\fN$, and a map
  $\sig \col \uOm\setminus\{\est\} \to \fN$, 
  we define the $\sig$-composition
\[ (M,U)\in\kk^\UN\times\kk^\uOm \mapsto M\circ_\sig U\in\kk^\uOm \]
by the formulas
\begin{align}
(M\circ_\sig U)^{\est} &\defeq M^{\est},  \\[1ex]
(M\circ_\sig U)^{\uom} &\defeq 
\sum_{s\geq 1} \, \sum_{
\substack{ \uom =\uom^1 \cdots \, \uom^s \\
\uom^1,\ldots,\, \uom^s \in \uOm\setminus\{\est\}
}}
M^{\sig(\uom^1) \cdots \sig(\uom^s)} U^{\uom^1} \cdots U^{\uom^s}
  \quad \text{for $\uom\in\uOm\setminus\{\est\}$.}
\end{align}
\end{definition}

Thus, we recover the `$\odot$' composition in the special case when
$\fN=\N^*$ and $\sigma(\uom)=r(\uom)$ (with arbitrary $\Om$),
and \'Ecalle's composition `$\circ$' when $\fN=\Om$ is a commutative
semigroup and $\sig(n_1\cdots n_r) = n_1+\cdots+n_r$
for any nonempty word of~$\UN$.
Some classical properties of the latter operation can be generalized as follows:
\begin{enumerate}[(i)]
\item \label{itemMtimesN}
$(M\circ_\sig U)\times(N\circ_\sig U)=(M\times N)\circ_\sig U$.
\item   \label{itemexplog}
$\ee^{M\circ_\sig U} = (\ee^M)\circ_\sig U$ if $M^\est=0$, \ens
$\log(M\circ_\sig U) = (\log M)\circ_\sig U$ if $M^\est=1$.
\item   \label{itemidm}
$I\circ_\sig U=U-U^{\est}\idm_\Om$, 
where $I$ is defined by~\eqref{eqdefI} and~$\idm_\Om$ is the unit of $\kk^\uOm$.
\item   \label{itemderivphipsi}
Denote by $\iota_\Om\col\Om\xhookrightarrow{}\uOm\setminus\{\est\}$ the
inclusion map.
If $\phi\col\fN\to\kk$ is a function such that
%
%
$\phi\circ\sig$ maps the concatenation in~$\uOm$ to the addition
in~$\kk$, then
\begla 
(\na_\phi M) \circ_\sig U  = \na_\psi (M \circ_\sig U)
\quad \text{for all $M\in \kk^\uN$},
\qquad \text{with $\psi \defeq \phi\circ\sig\circ\iota_\Om$,}
\edla
where $\na_\phi$ and $\na_\psi$ are the mould derivations
defined by~\eqref{eqdefnaphi}.
\item
If $U$ is alternal and
$\sig(\om_1\cdots\,\om_r) = \sig(\om_{\tau(1)}\cdots\om_{\tau(r)})$
for every permutation~$\tau$ and for any $\om_1,\ldots,\,\om_r\in\Om$, then
%
\[
\text{$M$ alternal} \IMP \text{$M\circ_\sig U$ alternal,} \qquad
\text{$M$ symmetral} \IMP \text{$M\circ_\sig U$ symmetral.}
\]
\item
Suppose $(B_\om)_{\om\in\Om}$ satisfies Assumption~\ref{assumpBn}.
Then the formula
$D_n \defeq \sum\limits_{\uom\in\sig\ii(n)} U^\uom B_\uom$
defines a family $(D_n)_{n\in\fN}$
which also satisfies Assumption~\ref{assumpBn}, and
\[
MD=(M\circ_\sig U)B
\quad\text{for any $M \in \kk^\UN$.}
\]
\item
  Suppose that $\tau\col\UN\setminus\{\est\}\to\fM$ is a map such that
  $\psi\defeq \tau\circ\iota_\fN\circ\sig$ satisfies
\begin{gather*}
\psi(\uom^1\cdots\,\uom^s)= \tau(\sig(\uom^1)\cdots\sig(\uom^s))
\ens\text{for any $s\ge1$ and $\uom^1,\ldots\,\uom^s\in\uOm\setminus\{\est\}$,}
\\[-1.5ex]
\intertext{then}
M\circ_{\psi}(N\circ_\sig U)=(M\circ_{\tau}N)\circ_\sig U
\quad\text{for any $M\in\kk^\UM, N\in \kk^\UN, U\in \kk^\uOm$.}
\end{gather*}
\end{enumerate}
(The proof of these properties is left to the reader.)


\begin{proof}[Proof of Theorem~\ref{thmSNodotUSOmT}]
Here $\Om = \{x,y\}$, $\fN=\N^*$ and $\sig = r\col \uOm \to\fN$ is
word length.
Since $T_\fN=\log S_\fN$ and $T_\Om = \log S_\Om$, in view
of~\eqref{itemexplog} it is sufficient to prove
$S_\fN\circ_\sig U = S_\Om$.

As noticed in Section~\ref{secAltMldExpPsi}, $S_\fN$ is a solution in
$\Q^\UN$ to equation~\eqref{eqmouldeq}, which involves $\na=\na_\phi$,
with the notation $\phi\col\fN\xhookrightarrow{}\Q$ for the inclusion
map.
Taking `$\odot U$' of both sides of~\eqref{eqmouldeq}, we get 
\beglab{eqnaphiSUIS}
(\na_\phi S_\fN) \circ_\sig U = (I\times S_\fN) \circ_\sig U.
\edla
We compute the \lhs\ by means of~\eqref{itemderivphipsi}:
$\phi\circ\sig(\uom) = r(\uom)$ is word length, in particular it maps
concatenation in~$\uOm$ to addition in~$\Q$, and
$\phi\circ\sig\circ\iota_\Om \equiv 1$,
hence the \lhs\ is $\na_1(S_\fN\circ_\sig U)$.
Note that the mould derivation~$\na_1$ is given by $(\na_1 M)^\uom =
r(\uom) M^\uom$.

By~\eqref{itemMtimesN} and~\eqref{itemidm}, the \rhs\
of~\eqref{eqnaphiSUIS} is
$(I\circ_\sig U)\times (S_\fN\circ_\sig U) = U \times (S_\fN\circ_\sig
U)$.
Therefore, $S_\fN\circ_\sig U$ is a solution to
\beglab{eqnaunMUM}
M^\est = 1, \qquad \na_1 M = U \times M.
\edla
It is easy to see that~\eqref{eqnaunMUM} has no other solution in~$\Q^\uOm$.

On the other hand, by~\eqref{eqMldExpSOm}, 
$S_\Om = \ee^{I_x} \times \ee^{I_y}$,
and $\na_1$ is a mould derivation which satisfies
$\na_1 I_x = I_x$ and $\na_1 I_y = I_y$,
thus
\begin{multline*}
\na_1 S_\Om = \na_1(\ee^{I_x}) \times \ee^{I_y} + \ee^{I_x} \times \na_1(\ee^{I_y})
= I_x \times \ee^{I_x} \times \ee^{I_y} + \ee^{I_x} \times I_y \times \ee^{I_y}\\[1ex]
= (I_x + \ee^{I_x} \times I_y \times \ee^{-I_x}) \times \ee^{I_x} \times \ee^{I_y}
= U \times S_\Om
\end{multline*}
by~\eqref{eqUAdeIx}.
Therefore $S_\Om$ is a solution to~\eqref{eqnaunMUM}, hence it must
coincide with $S_\fN\circ_\sig U$.
\end{proof}


\bigskip

\bigskip

\subsubsection*{Acknowledgments}

D.S.\ and Y.L.\ thank the Centro Di Ricerca Matematica Ennio De Giorgi and the Scuola Normale Superiore di Pisa for their kind hospitality, during which this work was completed.

\newpage


\end{document}